\documentclass[letter,12pt]{amsart}
\usepackage{amsmath, amsthm, amstext}
\usepackage{amssymb, array, amsfonts}
\usepackage[english]{babel}
\usepackage{bbding}
\usepackage{float}
\usepackage{tikz}
\usepackage{pdfpages}
\usepackage{hyperref}
\usepackage{float}
\usepackage{graphics}
\usepackage{graphicx}
   \usepackage{enumerate}
\usepackage[margin=2.2cm]{geometry}
\numberwithin{equation}{section}

\def \er{\varepsilon}
\def \ep{\varepsilon}

\def \F {\mathcal F}
\def \G {\mathcal G}

\renewcommand{\mod}{\,\mathrm{mod}\,}

\def \C{\mathbb{C}}

\def \M2{\mathrm{M}_2}
\def \R{\mathbb{R}}
\def \Z{\mathbb{Z}}

\def \T{\mathbb{T}}

\def \sl2r{\mathrm{SL}(2,\R)}

\newcommand{\beq}{\begin{equation}}
\newcommand{\eeq}{\end{equation}}

\def\sign{\operatorname{sign}}

\def\im{\operatorname{Im}}
\def\re{\operatorname{Re}}
\newcommand{\eqdef}{\stackrel{\rm def}{=\kern-3.6pt=}}
    \newcommand{\<}{\langle}
\renewcommand{\>}{\rangle}

\theoremstyle{plain}
\newtheorem{theorem}{\bf Theorem}[section]
\newtheorem{lemma}[theorem]{\bf Lemma}

\newtheorem{prop}[theorem]{\bf Proposition}
\newtheorem{cor}[theorem]{\bf Corollary}

\theoremstyle{definition}

\theoremstyle{remark}
\newtheorem{remark}[theorem]{\bf Remark}

\theoremstyle{cond}

\renewcommand{\le}{\leqslant}
\renewcommand{\ge}{\geqslant}

\renewcommand{\qed}{\vrule height7pt width5pt depth0pt}
\title{On spectral bands of discrete periodic operators}
\author[N. Filonov]{Nikolay Filonov}
\address{St.~Petersburg Department of V.~A.~Steklov Mathematical Institute,
	Fontanka 27,
 St.Petersburg, 191023, Russia, and
	St. Petersburg State University,
Universitetskaya emb. 7/9, St. Petersburg, 199034,
Russia}
\thanks{The first author was supported by RSF 22-11-00092.}
\author[I. Kachkovskiy]{Ilya Kachkovskiy}
\address{Department of Mathematics,
	Michigan State University,
	Wells Hall, 619 Red Cedar Road,
	East Lansing, MI, 48910,
	United States of America}
\thanks{The second author was partially supported by the National Science Foundation DMS--1846114 grant. 
The second author's research visits to St. Petersburg were partially supported by the RSF 17-11-01069 grant and the program "Spectral Theory and Mathematical Physics" at Euler Institute, St. Petersburg}
\begin{document}
\maketitle
\begin{abstract}
We consider discrete periodic operator on $\Z^d$ with respect to lattices $\Gamma\subset\Z^d$ of full rank. We describe the class of lattices $\Gamma$ for which the operator may have a spectral gap for arbitrarily small potentials. We also show that, for a large class of lattices, the dimensions of the level sets of spectral band functions at the band edges do not exceed $d-2$.

\end{abstract}
\section{Definitions and results}
\subsection{Lattices in $\Z^d$} Let $\Gamma\subset\Z^d$ be a lattice of full rank with a basis $a_1,\ldots,a_d\in \Z^d$ (linearly independent over $\R$):
$$
\Gamma=\{n_1 a_1+\ldots+n_d a_d\colon n_1,\ldots,n_d\in \Z\}.
$$
We will consider discrete Schr\"odinger operators on $\ell^2(\Z^d)$ with potentials $V$ periodic with respect to $\Gamma$:
\beq
\label{eq_h_def}
(H\psi)(n)=(\Delta\psi)(n)+V(n)\psi(n)=\sum_{m\colon |m-n|_{1}=1}\psi(m)+V(n)\psi(n),
\eeq
where $|n|_1=|n_1|+\ldots+|n_d|$ denotes the $\ell^1$ norm, and the potential $V$ is $\Gamma$-periodic:
$$
V(n+a)=V(n),\quad \forall a\in \Gamma.
$$
If $a_1,\ldots,a_d$ is a basis of a lattice $\Gamma$, its dual basis $b_1,\ldots,b_d$ is uniquely determined by the relations
$$
\<a_i,b_j\>=\delta_{ij},
$$
where $\delta_{ij}$ is the Kronecker symbol and $\<\cdot,\cdot\>$ is the standard inner product on $\R^d$. The {\it dual lattice} $\Gamma'$ (also called the {\it reciprocal lattice} in literature) is defined by
$$
\Gamma'=\{n_1 b_1+\ldots+n_d b_d\colon n_1,\ldots,n_d\in \Z^d\}\subset \R^d.
$$
One can easily check that $\Gamma'$ as a lattice does not depend on the particular choice of a basis $a_1,\ldots,a_d$ of $\Gamma$. Alternatively, the following is an equivalent coordinate-free definition:
$$
\Gamma'=\{b\in \R^d\colon e^{2\pi i \<b,a\>}=1,\,\forall a\in\Gamma\}.
$$
One can view $\Gamma'$ as the collection of all rational linear relations (modulo $\Z^d$) between the coordinates of vectors from $\Gamma$. In particular, since $\Gamma\subset\Z^d$, we have $\Gamma'\supset\Z^d$.

Some of our results will assume that $\Gamma$ is not of a certain type. We will say that $\Gamma$ is an {\it even lattice} if $(1/2,\ldots,1/2)\in \Gamma'$. In other words, $\Gamma$ is even if and only if for any $n=(n_1,\ldots,n_d)\in \Gamma$ we have $n_1+\ldots+n_d\in 2\Z$. A lattice $\Gamma$ is even if and only if the potential $V(n)=v(n_1+\ldots+n_d)$ is $\Gamma$-periodic for any $2$-periodic function $v$ on $\mathbb Z$. Such potentials will be called {\it checkerboard potentials}.

We will also say that $\Gamma$ is {\it divisible} if the following is true: for any $(\ep_1,\ldots,\ep_d)\in \{-1,1\}^{d}$ and any $j=1,2,\ldots,d$, there exists an integer $p_j\ge 2$ such that 
$$
\frac{1}{p_j}(\ep_1,\ep_2,\ldots\ep_{j-1}, \ep_j,\ep_{j+1},\ldots,\ep_d)\in \Gamma'\,\text{ or }\,\frac{1}{p_j}(\ep_1,\ep_2,\ldots\ep_{j-1},-\ep_j,\ep_{j+1},\ldots,\ep_d)\in \Gamma'.
$$
Note that if we have $p_j=2$ for at least one $j$, then the lattice will become even.

For $d=2$, $\Gamma$ is divisible if and only if $\Gamma'$ contains a vector $(1/p,1/p)$ or $(1/p,-1/p)$ for some integer $p\ge 2$. The second case can be reduced to the first case by reflection, and we can assume without loss of generality that $(1/p,1/p)\in \Gamma'$. In this case, $\Gamma$ is divisible if and only if the potential $V(n)=v(n_1+n_2)$ is $\Gamma$-periodic for any $p$-periodic function $v$ on $\mathbb Z$. Such potentials will be called $p$-{\it periodic potentials}. If $p=2$, this condition defines a checkerboard potential.

\subsection{Bloch wave expansion.} Let $V\colon \Z^d\to \R$ be a $\Gamma$-periodic potential. Denote by $\ell^2_{\theta}(\Z^d/\Gamma)$ the space of all complex-valued functions $\psi\colon \Z^d\to \C$ with {\it quasimomentum }$\theta$. By definition, the latter means
$$
\psi(n+a)=e^{2\pi i \<\theta,a\>}\psi(n),\quad \forall n\in \Z^d,\,\,a\in\Gamma.
$$
Clearly, if $\psi$ has quasimomentum $\theta$, then it also has quasimomentum $\theta+b$ for any $b\in \Gamma'$. Therefore, it is natural to consider the quasimomentum as an element of $\R^d/\Gamma'$ or, equivalently, an element of a representative set of $\R^d/\Gamma'$ such as the elementary cell
$$
\Omega'=\{\eta_1 b_1+\ldots+\eta_d b_d\colon \eta_1,\ldots,\eta_d\in [0,1)\}.
$$
Any function $\psi$ with quasimomentum $\theta$ is uniquely determined by its values on any representative set of $\Z^d/\Gamma$: for example, on
$$
\Lambda=\{x_1 a_1+\ldots+x_d a_d\colon x_1,\ldots,x_d\in [0,1)\}\cap\Z^d.
$$
In other words, we can identify the space $\ell^2_{\theta}(\Z^d/\Gamma)$ with $\ell^2(\Lambda)$, thus making it a finite-dimensional Hilbert space. With this identification, the classical Floquet--Bloch--Gelfand theorem states that the operator $H$ is unitarily equivalent to a direct integral:
\beq
\label{eq_dirint}
H\cong\int_{\Omega'}\oplus h(\theta)\,d\theta,
\eeq
where $h(\theta)$ is the ``restriction'' of the original operator $H$ to the space $\ell^2_{\theta}(\Z^d/\Gamma)$. By ``restriction'' we mean an operator defined by the same formula \eqref{eq_h_def}: 
$$
(H\psi)(n)=((\Delta(\theta)+V)\psi)(n)=(\Delta\psi)(n)+(V\psi)(n).
$$
Note that $\ell^2_{\theta}(\Z^d/\Gamma)$ is not a subspace of $\ell^2(\Z^d)$. To emphasize this, we use the notation $\Delta(\theta)$ for the (formal) action of the operator $\Delta$ on $\ell^2_{\theta}(\Z^d/\Gamma)$. 

Let 
$$
N=\#\Lambda=\# \Z^d/\Gamma=\dim \ell^2_{\theta}(\Z^d/\Gamma).
$$ 
Since $h(\theta)$ is a self-adjoint operator on $\ell^2_{\theta}(\Z^d/\Gamma)$, one can numerate its eigenvalues in the non-decreasing order:
$$
E_1(\theta)\le E_2(\theta)\le\ldots\le E_N(\theta).
$$
Each function $E_j(\cdot)$ is a real-valued Lipschitz and piecewise real analytic function of $\theta$. Since $\Omega'$ can be identified with the torus $\R^d/\Gamma'$, each function $E_j(\theta)$ can be extended into $\R^d$ by $\Gamma'$-periodicity. These functions are usually called the {\it spectral band functions} of the $\Gamma$-periodic operator $H$. Denote by
$$
[E_j^-,E_j^+]=\{E_j(\theta)\colon \theta\in \Omega'\}
$$
the range of $E_j$. The interval $[E_j^-,E_j^+]$ is called {\it the $j$th spectral band of the operator $H$}. It is well known that the spectrum of $H$ is the union of the spectral bands:
$$
\sigma(H)=\cup_{j=1}^N [E_j^-,E_j^+].
$$
We should note that the number of bands not only depends on the operator $H$ per se, but also on the choice of the lattice $\Gamma$ (since one can always take a sub-lattice of periods $\Gamma_1\subset\Gamma$ thus increasing the size of the elementary cell); however, it is not hard to re-calculate the band structure using a new lattice, and our results will not depend on this ambiguity. An example of such calculation will be later done for the free operator $H=\Delta$.

We refer to \cite{Ku_book,Ku_review,RS4} for more detailed overview and introduction to spectral theory of Schr\"odinger operators with periodic potentials, and to \cite[Chapter 10]{AMe} for some physical foundations of the tight binding approximation that leads to discrete Schr\"odinger operators.

\subsection{Main results} In this paper, we will only consider the case $d\ge 2$. Among some natural questions arising in the theory of periodic operators, we will single out the following.
\begin{enumerate}
	\item Can we show that there are no flat bands? In other words, is it true that $E_j^-<E_j^+$ for all $j$?
	\item Do the bands overlap? Can we show that the spectrum of $H$ is an interval for $\|V\|_{\infty}\ll 1$?
	\item What is the structure of the level sets $\{\theta\in \Omega'\colon E_j(\theta)=E_j^{\pm}\}$?
\end{enumerate}
Question (1) is equivalent to the question of absolute continuity of the spectrum of $H$ and is probably a folklore result for any lattice $\Gamma$. See, for example, \cite{Kruger} for the case of a rectangular lattice. It will also easily follow from the arguments of the present paper (in particular, Proposition \ref{prop_ac}).

Question (2) is often referred to as the discrete Bethe--Sommerfeld Conjecture. The original Bethe--Sommerfeld conjecture was established in the continuum in \cite{Parnovski1}, see also \cite{Parnovski2,Skriganov}, and states that the spectrum of a Schr\"odinger operator in the continuum has only finitely many gaps if $d\ge 2$; or, equivalently, that the spectrum does not have gaps located at sufficiently large energies. In the discrete setting, the kinetic energy operator $\Delta$ is bounded, and one can imitate large kinetic energy regime by considering small potentials, which leads to Question (2). For rectangular lattices, the corresponding result was proved in \cite{Han} for small potentials, see also \cite{Kruger,Fillman}, assuming that $\Gamma$ has at least one odd period. As it turns out, a checkerboard potential is a counterexample which applies to the case when all periods of $\Gamma$ are even: the operator $\Delta+\er V$ has a gap in the middle of the spectrum for all $\er>0$.

Until recently, Question (3) has only been addressed in the continuous setting. In \cite{FK_acta}, it was shown that, for Schr\"odinger operators on $\mathbb R^2$, 
 the level sets mentioned in (3) are finite. Numerous results were earlier established for the ground state, see for example, \cite{KSi,BSu03}. See also some related results in \cite{V,KL,Sht2,Sht3}. In the tight binding setting, the equivalent of \cite{FK_acta} is false with the same counterexample as (2): the checkerboard potential. Some positive results (both in the discrete and continuous setting for $d=2$) have been obtained in \cite{PS}: it is shown that one can construct a sequence of small perturbations of larger and larger periods to lower the degeneracy of the band edge and ultimately thansform the band function into a Morse function (however, one cannot treat all bands simultaneously by this method, as their number grows each time a perturbation is applied). One should note that (3) is a part of the ``effective mass conjecture'' which states that for ``generic'' potentials all band functions behave like Morse functions around their global minima and maxima. Also, a proper analogue of this conjecture needs to be formulated in view of the known counterexamples. 
  
 The following are main results of the present paper.
\begin{theorem}
\label{th_main_edges}
Suppose that $\Gamma\subset \Z^d$ is not divisible. Then each set $\{\theta\in \Omega'\colon E_j(\theta)=E_j^{\pm}\}$ has dimension at most $d-2$.
\end{theorem}
\begin{remark}
It is easy to see that, in appropriately chosen coordinates, the above set is algebraic. In particular, it can be represented as a finite union of smooth submanifolds of $\mathbb T^d$, and therefore any reasonable notion of dimension can be used.
\end{remark}
\begin{theorem}
\label{th_main_bz}
Suppose that $\Gamma\subset \Z^d$ is not even. Then there exists $\er=\er(\Gamma)>0$ such that, for any $\Gamma$-periodic potential $V$ with $\|V\|_{\infty}<\er$, the spectrum of $H$ is an interval.
\end{theorem}
One can think that Theorem \ref{th_main_bz} is an incremental improvement over \cite{Han}. However, our proof is different from \cite{Han}, and some of the arguments are also used in Theorem \ref{th_main_edges}. One can also note that any lattice $\Gamma$ of rank $d$ contains a rectangular sub-lattice $\Gamma_1$ of rank $d$. Any $\Gamma$-periodic potential is also $\Gamma_1$-periodic, and therefore one may be tempted to simply apply the result of \cite{Han} to $\Gamma_1$. However, it may happen that all such $\Gamma_1$ are checkerboard lattices: for example, consider a non-checkerboard lattice $\Gamma\subset\Z^2$ with basis $\{(2,0),(1,2)\}$. Therefore, Theorem \ref{th_main_bz} does not follow from \cite{Han} directly.

The example of a checkerboard potential shows that one cannot extend Theorems \ref{th_main_edges} and \ref{th_main_bz} to the case of even lattices. For the case of Theorem \ref{th_main_bz}, it was observed in \cite{Han}, \cite{Fillman}. For the case of Theorem \ref{th_main_edges}, it was shown in \cite{FK_acta} (and probably known before) that checkerboard potentials in $d=2$ produce Fermi surfaces of dimensions $d-1$ at the edges of the spectral gap. This calculation can easily be extended to arbitrary dimension.

In the case $d=2$, Theorem \ref{th_main_edges} is sharp: the $p$-periodic potentials also produce one-dimensional Fermi surfaces at the edges of some bands. A simple argument has been provided in \cite{PS}.

The case of divisible lattices for $d\ge 3$ is currently open. It is easy to see that the ``minimal'' potential counterexample would be $\Gamma=(3\Z)^3$ whose fundamental domain has 27 points.

 Theorem \ref{th_main_edges} is strongly related to the question of irreducibility of the Fermi varieties. For discrete Schr\"odinger operators with rectangular lattices $\Gamma$ and coprime periods, irreducibility was shown in \cite{GKT} for $d=2$ and in \cite{Battig} for $d=3$. Recently, a breakthrough in that direction was obtained in \cite{Liu}, where irreducibility was established in all dimensions for rectangular lattices with coprime periods. For such lattices, Theorem \ref{th_main_edges} is a corollary of irreducibility. Our method of proof of Theorem \ref{th_main_edges} is substantially different and does not imply irreducibility.

\section{The real Fermi surface and the proof of Theorem \ref{th_main_bz}}
Most of the paper deals with the case of the free Laplacian; that is, the case $V\equiv 0$. Let
$$
F(\theta):=2\cos(2\pi \theta_1)+\ldots+2\cos(2\pi \theta_d),\quad \theta=(\theta_1,\ldots,\theta_d)\in \R^d.
$$
It is well known that the Laplace operator $\Delta$ on $\ell^2(\Z^d)$ is unitarily equivalent to an operator of multiplication by $F(\theta)$ on $L^2([0,1)^d)$. 
Denote by
$$
e_{\theta}(n)=e^{2\pi i\<n,\theta\>}.
$$
The above unitary equivalence can be formally written as
$$
(\Delta(\theta) e_{\theta})(n)=F(\theta)e_{\theta}(n),
$$
and therefore $e_{\theta}$ can be considered as a generalized eigenfunction of the absolutely continuous spectrum for $\Delta$. 
We will call $e_{\theta}$ the standard (generalized) eigenfunction of $\Delta$ with the momentum $\theta$. These eigenfunctions are parametrized by $\theta\in[0,1)^d$. We would like to relate it with the Floquet expansion \eqref{eq_dirint}. Let 
$$
\Lambda'=\Gamma'/\Z^d.
$$ 
We will use the following natural convention: if $b\in \Lambda'$, then any equality that involves $b$ or components of $b$, will be considered modulo $\Z^d$. For example, we will often use notation like $b=0$, or $b\in \Lambda'\setminus\{0\}$, meaning in the latter case $b\neq 0\,\,\mathrm{mod}\,\,\Z^d$. For the components of $b$ (say, $b_1=1/2$), the equalities are considered modulo $\Z$.

Note that $\#\Lambda'=\#\Lambda=N$, and there are exactly $N$ distinct standard eigenfunctions of $\Delta$ who have quasimomentum $\theta$. Moreover, one can easily check that these functions span the space $\ell^2_{\theta}(\R^d/\Gamma)$ of all functions with quasimomentum $\theta$:
$$
\ell^2_{\theta}(\Z^d/\Gamma)=\mathrm{span}\{e_{\theta+b}\colon b\in \Lambda'\}.
$$
It will be convenient for us to work in the above basis (which is orthogonal, and each element has norm $\sqrt{N}$). This basis also identifies $\ell^2_{\theta}(\Z^d/\Gamma)$ with $\ell^2(\Lambda')$.

The operator $h(\theta)=\Delta(\theta)+V$, therefore, is unitarily equivalent to the following operator:
\beq
\label{eq_hhat_def}
(\hat{h}(\theta)\varphi)(b)=F(\theta+b)\varphi(b)+\sum_{c\in \Gamma'}\hat{V}(b-c)\varphi(c).
\eeq
Here $\varphi\in \ell^2(\Lambda')$, and the difference $b-c$ is considered in $\Gamma'/\Z^d$ which is identified with $\Lambda'$. $\hat{V}$ is the discrete Fourier transform of $V$:
$$
\hat V(b)=\frac{1}{\sqrt{N}}\sum_{n\in\Lambda} V(n)e_b(n),\quad b\in \Lambda'.
$$
\subsection{The real Fermi surface} In this section, we will identify $[0,1)^d$ with the $d$-dimensional torus $\T^d$. Let $E\in \R$. Define
\beq
\label{eq_fermi_def}
\F_E=\{\theta\in \T^d\colon F(\theta)=E\}.
\eeq
The subset $\F_E$ is usually called the {\it real Fermi surface of the free Laplacian on $\Z^d$}. Denote also the (real) singular set
$$
S_E=\{\theta\in \F_E\colon (\nabla F)(\theta)=0\}=\{0,1/2\}^d\cap\F_E.
$$
Clearly, $S_E$ is non-empty only for $E\in\{-2d,-2(d-1),\dots,2(d-1),2d\}$. The set $\F_E$ is empty for $|E|>2d$ and consists of a single point for $|E|=2d$.
\begin{lemma}
\label{lemma_connected}
Suppose that $d\ge 3$ and $E\in (-2d,2d)$, or $d=2$ and $E\in(-4,0)\cup(0,4)$. Then $\F_E\setminus S_E$ is a smooth connected $d-1$-dimensional real analytic submanifold of $\T^d$.
\end{lemma}
\begin{proof}
The fact that $\F_E\setminus S_E$ is a submanifold follows from the implicit function theorem. To show that $\F_E\setminus S_E$ is connected, consider the map
$$
f\colon \F_E\setminus S_E \to [-1,1]^d,\,\,(\theta_1,\ldots,\theta_d)\mapsto (\cos 2\pi \theta_1,\cos 2\pi \theta_2,\ldots,\cos 2\pi \theta_d).
$$
It is easy to see that $f(\F_E\setminus S_E)$ is connected. Let $a,b\in \F_E\setminus S_E$. Since there is a continuous path between $f(a)$ and $f(b)$, it can be lifted to a continuous path between $a,b'\in f(\F_E\setminus S_E)$ where $f(b')=f(b)$. It remains to explain how to connect $b$ with $b'$, for which it is sufficient to construct a path, say, from  $(\theta_1,\ldots,\theta_d)$ to $(-\theta_1,\ldots,\theta_d)$ (assuming that both points are on $\F_E$). Since $|E|<2d$, by a small continuous perturbation one can assume that $|\cos 2\pi \theta_j|<1$ for all $j$. In this case, one can fix the coordinates $(\theta_3,\ldots,\theta_d)$ and consider the section by the two-dimensional hyperplane:  
$$
2\cos2\pi \theta_1+2\cos 2\pi \theta_2=E-F(\theta_3,\ldots,\theta_d).
$$
By assumption, the right hand side is contained in $(-4,4)$. If it is not zero, the equation defines a smooth connected curve in the coordinates $\theta_1,\theta_2$, where one can easily connect $(\theta_1,\theta_2)$ with $(-\theta_1,\theta_2)$ (still, staying within $\F_E\setminus S_E$). If the right hand side is $0$, one can perturb both $\theta_2$ and $\theta_3$ to make it non-zero. The only case when it is not possible is when there is no $\theta_3$, that is, $d=2$ and $E=0$.\,\end{proof}
\begin{remark}
\label{rem_structure_fe}
Let $d=2$. For $|E|>4$, the set $\F_E$ is empty. For $|E|=4$, $\F_E$ consists of a single point $(0,0)\in \T^2$ or $(1/2,1/2)\in \T^2$. For $0<|E|<4$, the set $\F_E$ is a one-dimensional simple closed real analytic curve surrounding an oval-shaped area with center at $(0,0)$ for $E>0$ and $(1/2,1/2)$ for $E<0$. If $E=0$, the set $\mathcal F_0$ is a union of four intervals whose interiors do not intersect. The set $\mathcal F_0\setminus S_0$ has four connected components.
\end{remark}
\begin{lemma}
\label{lemma_symmetry}
Let $d=2$. Suppose that $E_1,E_2\in (0,4)$ and $\F_{E_1}+b=\F_{E_2}$ for some $b\in \T^2$. Then $b=0$ and $E_1=E_2$.
\end{lemma}
\begin{proof}
Consider both $\F_{E_1}$ and $\F_{E_2}$ as $\Z^2$-periodic subsets of $\R^2$. Since each connected component of $\F_{E_1}$ and $\F_{E_2}$ is an analytic curve, we must have $\F_{E_1}+b=\F_{E_2}$. Since the barycenter of each connected component of $\F_{E_1}$ must be translated into the barycenter of some connected component of $\F_{E_2}$, we must have $b\in \Z^2$ and therefore, without loss of generality, $b=0$ and $\F_{E_1}=\F_{E_2}$. In view of the definition of $\F_E$, the latter is only possible for $E_1=E_2$.
\end{proof}
\begin{cor}
\label{cor_symmetry}
Let $d=2$ and $E_1,E_2\in \R$. Suppose that $\F_{E_1}\neq \varnothing$ and $\F_{E_1}+b=\F_{E_2}$ for some $b\in \T^2$. Then either $b=0$ and $E_1=E_2$, or $b=(1/2,1/2)$ and $E_1=-E_2$.	
\end{cor}
\begin{proof}
From Remark \ref{rem_structure_fe}, it follows that $E=0$ is the only case for which $\F_E$, considered as a $\Z^2$-periodic subset of $\R^2$, is non-empty and connected. Therefore, $E_1=0$ if and only if $E_2=0$, and the statement of the corollary follows in this case. Since $\F_{E_1}=\F_{-E_1}+(1/2,1/2)$, the case $0<|E_1|<4$ follows from Lemma \ref{lemma_symmetry}. If $|E_1|=4$, then $\F_{E_1}=\Z^2$ or $\F_{E_1}=\Z^2+(1/2,1/2)$, therefore we must also have $|E_2|=4$ and the conclusion follows. 
\end{proof}
Let $\F_E$ be a Fermi surface \eqref{eq_fermi_def}. We will call a point $\theta\in \F_E$ {\it generic} if $\theta+b\notin \F_E$ for any $b\in \Lambda'$, $b\neq 0 \,\mod\, \Z^d$. Except for the special case $d=2$, $E=0$, almost every point of $\F_E$ is generic. The proof below is based on the following observation: if $\F_E$ has too many non-generic points, then it has a non-trivial translational symmetry by a vector from $\Lambda'$. This symmetry generates a symmetry on two-dimensional slices. Each slice is a two-dimensional Fermi surface, but the symmetry does not have to preserve the energy of a slice. However, Lemma \ref{lemma_symmetry} implies that, even for slices with different energies, possible (non-trivial) symmetries that transform one into another, are restricted to vectors whose components are equal to $1/2$.
\begin{theorem}
\label{th_generic}
Let $\F_E$ be a Fermi surface \eqref{eq_fermi_def} with $d\ge 3$ or $d=2$ and $E\neq 0$. Suppose that $\Gamma$ is not even $($that is, suppose that $(1/2,\ldots,1/2)\notin \Gamma'$ $)$. Then the set of non-generic points of $\F_E$ is contained in a finite union of $(d-2)$-dimensional real analytic submanifolds of $\T^d$. As a consequence, the set of generic points is open in $\F_E$ and has full $(d-1)$-dimensional measure in $\F_E$.
\end{theorem}
\begin{proof}
For each $b\in \Lambda'\setminus\{0\}$, consider the function $F(\theta+b)-E$. Clearly, this function is real analytic on a connected real analytic manifold $\F_E\setminus S_E$ (in fact, it is an rational function in appropriate coordinates). As a consequence, it is either identically zero, or its zero set is contained in a finite union of $(d-2)$-dimensional submanifolds. In order to prove the theorem, it would be sufficient to show that none of the functions $F(\theta+b)-E$ is identically zero.

Suppose that $F(\theta+b)\equiv E$ on $\F_E\setminus S_E$. By continuity, the same also holds on $\F_E$, and therefore we have $\F_E+b=\F_E$. Fix some $\theta'\in \T^{d-2}$ and denote by
$$
P(\theta')=\{\theta\in \T^d\colon \theta=(\theta_1,\theta_2,\theta')\}
$$
the two-dimensional torus obtained by fixing all coordinates of $\theta$ except for the first two. Suppose that $\F_E\cap P(\theta')\neq \varnothing$. Let $b=(b_1,b_2,b')$ and define
$$
\G_j=\{(\theta_1,\theta_2)\colon 2\cos 2\pi \theta_1+2\cos 2\pi \theta_2=E_j\},\quad j=1,2,
$$
where
$$
E_1=E-F(\theta'),\quad E_2=E-F(\theta'+b').
$$
Since $\F_E=\F_E+b$, we have
$$
\G_1+(b_1,b_2)= \G_2.
$$
Apply Lemma \ref{lemma_symmetry} with the above $E_1,E_2$, and conclude that $b_1=b_2=1/2$ or $b_1=b_2=0$. Now, without loss of generality we could have chosen any two components of $\theta$ instead of $\theta_1$ and $\theta_2$, and hence we have $b_1=b_2=\ldots=b_d$. The case $b_1=b_2=\ldots=1/2$ is forbidden due to the assumption on $\Gamma$, which completes the proof.
\end{proof}
\begin{remark}
\label{rem_proff_d2}	
While the proof of Theorem \ref{th_generic} works for $d=2$ if $E\neq 0$, it fails for the special case   $d=2$, $E=0$, if $\Gamma'$ contains vectors of the form $(a,a)$ or $(a,-a)$ (in other words, if $\Gamma$ is $p$-divisible for some $p\ge 3$). However, one can modify the argument as follows, in a way similar to \cite{Han}.
\end{remark}
\begin{prop}
\label{prop_2d_generic}
Assume that $d=2$ and $E=0$, and $\Gamma$ is not even. Then, almost every point $\theta\in \mathcal F_0$ satisfies the following:
\begin{enumerate}
	\item $\#\{b\in \Lambda'\colon \theta+b\in \F_0\}$ is odd.
	\item For any $b\in \Lambda'$, we have $(\nabla F)(\theta+b)\neq 0$.
\end{enumerate}
\end{prop}
\begin{proof}
First, note that (2) holds for every point on $\mathcal F_0$ which is not a translation of a corner point by a vector $b\in \Lambda'$. Therefore, we only need to establish (1). Recall that $\mathcal F_0$ is a union of four line segments, each of which is parallel to $(1,1)$ or $(1,-1)$. As a consequence, for almost every point $\theta\in \mathcal F_0$ and any vector $b\in \Lambda'$ that is not of the form $(a,a)$ or $(a,-a)$, we would have $\theta+b\notin \mathcal F_0$. In other words, after removing finitely many points from $\mathcal F_0$, we can assume, for the purposes of establishing (1), that $\Lambda'$ only contains vectors of the form $(a,a)$ and $(a,-a)$.

Under the above assumptions, the case $a=1/2p$ is not allowed since $\Gamma$ is not even. If $\theta$ belongs to the interior of a line segment of $\mathcal F_0$ parallel to $(1,1)$, the cardinality of the set in (1) becomes equal to  the total number of vectors of the form $(a,a)$ in $\Gamma'$, which is odd. Similarly, for $\theta$ in the interiors of the remaining line segments, one needs to consider vectors of the form $(a,-a)\in\Gamma'$; again, the total number of such vectors is odd. In fact, one can check that (1) holds for all points $\theta\in \mathcal F_0$ including the corner point if one only counts vectors of the form $(a,a)$ and $(a,-a)$.
\end{proof}

\subsection{Proof of Theorem \ref{th_main_bz}} After all preparations, the rest of the proof is standard as discussed in \cite{Kruger,Fillman,Han}. In order to prove Theorem \ref{th_main_bz}, it would be sufficient to show that any $E\in (-2d,2d)$ is in the interior of some spectral band of $\Delta$. Indeed, if $E_j$ is in some gap of the spectrum of $\Delta+\ep_j V$ and $\ep_j\to 0+$ as $j\to \infty$, one can by choosing a subsequence assume that, for some fixed $1\le k\le N-1$, we have
$$
E_k^+(\ep_j)\le E_j\le E_{k+1}^-(\ep_j),
$$
where $E_k^{\pm}(\ep_j)$ denote the spectral band edges of $\Delta+\ep_j V$. By choosing a further subsequence, we can additionally assume that $E_j\to E\in [-2d,2d]$. Since the spectral band edges are continuous in $\ep$, we have
$$
E_k(0)^+\le E\le E_{k+1}^-(0).
$$
In other words, $E$ must be between two consecutive spectral bands of the free Laplacian $\Delta$ with non-overlapping interiors, and therefore $E\in (-2d,2d)$ since $\Delta$ has no flat bands.

Let $E\in (-2d,2d)$. The band functions of the free Laplacian $\Delta$ have the form
$$
\{E_1(\theta),\ldots,E_N(\theta)\}=\{F(\theta+b)\colon b\in \Lambda'\}.
$$
As a consequence, the following is true for the eigenvalue counting functions of the operators $\Delta_{\theta}$:
\beq
\label{eq_count_def}
N(\theta,E)=\#\{j\colon E_j(\theta)\le E\}=\#\{b\in \Lambda'\colon F(\theta+b)\le E\}.
\eeq
Recall that $\sigma(\Delta)$ is an interval. If $E$ is not in the interior of any band, then the counting function must be a constant in $\theta$:
$$
\#\{b\in \Lambda'\colon F(\theta+b)\le E\}=\mathrm{const}.
$$
Therefore, the following is sufficient: for any $E\in (-2d,2d)$, find $\theta_1$, $\theta_2$ such that  $N(\theta_1,E)\neq N(\theta_2,E)$. 

First, let us consider the case $d\ge 3$ or $E\neq 0$. Let $\theta$ be a generic point in the sense of Theorem \ref{th_generic} such that $(\nabla F)(\theta)\neq 0$, and let $u$ be any direction satisfying $\<u,(\nabla F)(\theta)\>\neq 0$. Clearly, for $|t|\ll 1$, the function $t\mapsto F(\theta+t u)-E$ is strictly monotone in $t$ and vanishes at $t=0$. On the other hand, $F(\theta+b+tu)\neq E$ for all $b\in \Lambda'\setminus \{0\}$. As a consequence,
$$
N(\theta+t u,E)=N(\theta-t u,E)\pm 1
$$
for $0<t\ll 1$, which provides required points $\theta_1,\theta_2$.

In the case where Theorem \ref{th_generic} is not applicable (that is, $E=0$ and $d=2$), one can use Proposition \ref{prop_2d_generic} and observe that if one passes any $\theta$ from that proposition along any direction which is not perpendicular to $(\nabla F)(\theta+b)$ for all $b\in \Lambda'$, then the counting function changes by an odd number, and therefore also cannot be a constant. The latter is a simple version of the ``perturb and count'' argument from \cite{Han}.\,\qed
\section{The complex Fermi surface and the proof of Theorem \ref{th_main_edges}}
Since $h(\theta)$ is an operator on a finite-dimensional Hilbert space, all questions about spectral bands, essentially, become questions about some concrete algebraic equations (depending on $\Gamma$ and $V$). In particular, spectral band functions have the following description, assuming $\theta\in \T^d$:
$$
\exists j\colon E=E_j(\theta) \quad \text{ if and only if }\quad \det(h(\theta)-E)=0.
$$
The determinant in the right hand side is a real analytic function of $\theta\in \T^d$. In fact, it also extends to a complex analytic function of $\theta\in \C^d$, which is $\Gamma'$-periodic. Let $z=e^{2\pi i \theta_1}$. One can easily check (using, for example, the explicit formula \eqref{eq_hhat_def}) that $\det(h(\theta)-E)$ is a Laurent polynomial in $z$, whose coefficients are complex analytic functions of $\theta_2,\ldots,\theta_d,E$.

The following result is well known and implies that $\Delta+V$ has no flat bands.
\begin{prop}
\label{prop_ac}	
Fix $(\theta_2,\ldots,\theta_d)\in \C^{d-1}$ and $E\in \mathbb C$. Then the set
$$
\{\theta_1\in \C\colon \det(h(\theta_1,\ldots,\theta_d)-E)=0\}
$$
is finite modulo translations $\theta_1\mapsto \theta_1+1$.
\end{prop}
\begin{proof}
The set under consideration is invariant under translations $\theta_1\mapsto\theta_1+1$, and is also a zero set of an entire function. Therefore, it is sufficient to show that $\det(h(\theta)-E)$ is not identically zero as a function of $\theta_1$. Since $|\cos 2\pi\theta_1|\to +\infty$ as $|\im \theta_1|\to +\infty$ uniformly in $\re \theta_1$, we have that the absolute values of all eigenvalues of $\Delta(\theta)$ approach $+\infty$ as $|\im\theta_1|\to +\infty$ uniformly in $\re\theta_1$. Since $\Delta(\theta)$ is a normal operator and $h(\theta)-E$ is a bounded perturbation of $\Delta(\theta)$, the same holds for the eigenvalues of $h(\theta)-E$, and therefore for the determinant.
\end{proof}
\begin{prop}
\label{prop_degenerate_root}
Suppose that $E_j(\theta)=E^{\pm}_j$. Fix $\theta_2,\ldots,\theta_d$ and consider the function 
\beq
\label{eq_function_f}
f(\eta)=\det (h(\eta,\theta_2,\ldots,\theta_d)-E_j^{\pm})
\eeq 
as an analytic function of $\eta$. Then $\eta=\theta_1$ is a zero of $f$ of multiplicity at least two. A similar statement holds if one considers $f$ as a Laurent polynomial in $z=e^{2\pi i \eta}$.
\end{prop}
\begin{proof}
The argument repeats a similar argument from \cite{FK_acta}. Fix $\theta'=(\theta_2,\ldots\theta_d)$ and consider $E_j(\cdot,\theta')$ as a function of its first argument. From Proposition \ref{prop_ac}, that function cannot be constant in $\theta_1$ on any interval. Since it attains its maximal or minimal value at $E_j^{\pm}$, the equation $E_j(\eta,\theta')=E_{j}^{\pm}\mp\delta$, considered as an equation in $\eta\in \R$, has at least two solutions near $\theta_1$ for sufficiently small $\delta>0$. As a consequence, the root $\eta=\theta_1$ of the function $f$ \eqref{eq_function_f} splits into two distinct roots after a small perturbation of the parameter $E_j^{\pm}$, and as a consequence is of multiplicity at least two.
\end{proof}
Clearly, any Laurent polynomial can be multiplied by an appropriate power of $z$ and made a regular polynomial (with no additional roots). Therefore, the condition of having roots of multiplicity two or higher is an analytic condition which can be formulated in terms of its discriminant. Recall that, for a monic polynomial
$$
p(z)=z^n+a_{n-1}z^{n-1}+\ldots+a_0
$$
with roots $z_1,\ldots,z_n$, its {\it discriminant} is defined as
$$
\Delta(p)=\prod_{1\le i<j\le n}(z_i-z_j)^2.
$$
It is clear that $\Delta(p)$ vanishes if and only if $p$ has roots
of multiplicity greater than or equal to 2. It is well known (see,
for example, \cite[Section 5.9]{WW}) that $\Delta(p)$ is a
polynomial function of the coefficients $a_0,\ldots,a_{n-1}$. 
\begin{prop}
\label{prop_discriminant}
Fix $E$, $(\theta_2,\ldots,\theta_d)$, and fix $\mu_2,\ldots,\mu_d\in \C$. Consider the following family of functions of $\theta_1$:
$$
\det(h(\cdot,\theta_2+t\mu_2,\theta_3+t\mu_3,\ldots,\theta_d+t \mu_d)-E),\quad t\in \C.
$$
Then the set of $t\in \C$ for which the function in the left hand side has a root of multiplicity $\ge 2$, is discrete in $\C$ or is equal to the whole $\C$.
\end{prop}
\begin{proof}
Rewrite the left hand side as a Laurent polynomial in $z$ and multiply to the appropriate power of $z$. Note that, since the coefficient at the lowest power of $z$ may sometimes vanish, this multiplication may result in adding a degenerate root $z=0$. However, since the said coefficient is analytic in $t$, this itself can only happen for a discrete set of values of $t$. Modulo this remark, the set of $t$ under consideration is a zero set of the discriminant of some polynomial whose coefficients are analytic in $t$, from which the claim of the proposition follows.
\end{proof}
\begin{lemma}
\label{lemma_separation}
Fix a lattice $\Gamma\subset\Z^d$. Then there exists a subset $S\subset \T^{d-1}$, $\dim S\le d-2$, such that the following is true. For any vector $u=(\ep_2,\ldots,\ep_{d})\in \{-1,1\}^{d-1}$ such that $\Gamma'$ does not contain vectors of the form $\frac{1}{p}(\pm1,\ep_2,\ldots,\ep_d)$ with integer $p\ge 2$ and for any $\theta'=(\theta_2,\ldots,\theta_d)\in \T^{d-1}\setminus S$, the following holds. One can find $t_0=t_0(\theta')>0$ such that, for all $t\ge t_0$, if for some $\theta_1\in \mathbb C$ we have 
$$
|2\cos 2\pi \theta_1 +F(\theta'-it u)|\le e^{\pi t},
$$
then, for any $b=(b_1,b')\in \Lambda'\setminus\{0\}$, we have
$$
|2\cos 2\pi (\theta_1+b_1)+F(\theta'+b'-it u)|> e^{\pi t}.
$$
\end{lemma}
\begin{proof}
Since there are at most $2^{d-1}$ possible choices for a vector $u$ with the stated properties, we can without loss of generality assume that $u\in \{-1,1\}^{d-1}$ is fixed. It will also be convenient to use $O$-notation in some estimates. One can check that the constants in these estimates are bounded by, say, $2d$. The following observation will be important: if $x,y\in \R$ and $\ep=\sign(y)$, then
\beq
\label{eq_cosine_full}
\cos (x+iy)=\frac12 e^{|y|}e^{- i\ep x}+O(e^{-|y|}).
\eeq
Suppose, by contradiction, that there exists a sequence $t_j\to +\infty$ and a vector $b=(b_1,b')\in \Lambda'$, $b\neq 0$, such that, for some $\theta_1=\theta_1(j)$ we have
\beq
\label{eq_eigen_close1}
|2\cos 2\pi \theta_1+F(\theta'-it_j u)|\le e^{\pi t_j},
\eeq
\beq
\label{eq_eigen_close2}
|2\cos 2\pi (\theta_1+b_1)+F(b'+\theta'-it_j u)|\le e^{\pi t_j};
\eeq
Assume first that $\theta_1=x+iy$, where $x,y\in \R$. Using \eqref{eq_cosine_full}, we have
\beq
\label{eq_eigen_close3}
F(\theta'-it_j u)=e^{2\pi t_j}\{e^{2\pi i \ep_2 \theta_2}+\ldots+e^{2\pi i \ep_d \theta_d}\}+O(e^{-2\pi t_j}),
\eeq
\beq
\label{eq_eigen_close4}
F(\theta'+b'-it_j u)=e^{2\pi t_j}\{e^{2\pi i \ep_2  (\theta_2+b_2)}+\ldots+e^{2\pi i \ep_d (\theta_d+b_d)}\}+O(e^{-2\pi t_j}).
\eeq
Let $\ep=\sign(y)$. Then, by combining \eqref{eq_eigen_close1}, \eqref{eq_eigen_close2}, and \eqref{eq_cosine_full}, we have
$$
e^{2\pi|y|}e^{-2\pi i\ep x}+e^{2\pi t_j}\{e^{2\pi i \ep_2 \theta_2}+\ldots+e^{2\pi i \ep_d \theta_d}\}=O(e^{\pi t_j}),
$$
$$
e^{2\pi |y|}e^{-2\pi i\ep (x+b_1)}+e^{2\pi t_j}\{e^{2\pi i \ep_2(\theta_2+b_2)}+\ldots+e^{2\pi i \ep_d (\theta_d+b_d)}\}=O(e^{\pi t_j}).
$$
Multiply the second equality by $e^{2\pi i \ep b_1}$ and subtract one from another. We obtain
$$
e^{2\pi t_j}\{e^{2\pi i \ep_2 \theta_2}+\ldots+e^{2\pi i \ep_d \theta_d}\}-e^{2\pi t_j}e^{2\pi i \ep b_1}\{e^{2\pi i \ep_2 (\theta_2+b_2)}+\ldots+e^{2\pi i \ep_d (\theta_d+b_d)}\}=O(e^{\pi t_j})
$$
as $t_j\to +\infty$. Clearly, this can only happen if 
\beq
\label{eq_more_theta_restriction}
e^{2\pi i \ep_2 \theta_2}+\ldots+e^{2\pi i \ep_d \theta_d}=e^{2\pi i \ep b_1} e^{2\pi i\ep_2 (\theta_2+b_2)}+\ldots+e^{2\pi i\ep b_1} e^{2\pi i \ep_d (\theta_d+b_d)}.
\eeq
As a consequence, if
$$
\theta'\notin S:=\{\theta'\in \T^{d-1}\colon \eqref{eq_more_theta_restriction}\text{ holds for some }b=(b_1,b')\in \Lambda'\setminus \{0\}\text { and some }\ep\in\{1,-1\}\},
$$
then the conclusion of the lemma is true. Due to the definition of $S$, we have $\dim S\le d-2$, unless \eqref{eq_more_theta_restriction} is satisfied for all $\theta'\in \T^{d-1}$ for some fixed $\ep$ and $b$. The latter, however, would imply
$$
-\ep b_1=\ep_2 b_2=\ep_3 b_3=\ldots=\ep_d b_d,
$$
or, equivalently, $b_1(-\ep,\ep_2,\ep_3,\ldots,\ep_d)\in \Gamma'$. Together with the inclusion $\Gamma'\supset\Z^d$, we arrive to $p^{-1}(-\ep,\ep_2,\ep_3,\ldots,\ep_d)\in \Gamma'$, which contradicts the assumption on $\Gamma$.
\end{proof}
\begin{cor}
\label{cor_separation}
Under the assumptions of Lemma $\ref{lemma_separation}$, suppose that, in addition,
$$
\theta'\notin\tilde{S}:=\left(\bigcup_{b=(b_1,b')\in \Lambda'}(S+b')\right)\cup \left(\bigcup_{b=(b_1,b')\in \Lambda'}\{\theta'\in\T^{d-1}\colon e^{2\pi i \ep_2(\theta_2+b_2)}+\ldots+e^{2\pi i \ep_d (\theta_d+b_d)}=0\}\right).
$$
Then
$$
|F(\theta'+b'-itu)|\ge e^{3\pi t/2},\,\,\forall t\ge t_0(\theta'),\,\,\forall b=(b_1,b')\in \Lambda'.
$$
Moreover, for any $\theta_1\in \C$, the inequality
$$
|2\cos 2\pi(\theta_1+b_1)+F(\theta'+b'-itu)|\le e^{\pi t}
$$
may hold for at most one choice of $b=(b_1,b')\in \Lambda'$.
\end{cor}
\begin{proof}
The first claim follows from the definition of $\tilde{S}$ and \eqref{eq_eigen_close3}. The second claim follows from Lemma \ref{lemma_separation} applied to $\theta+b$ for all possible $b\in \Lambda'$.
\end{proof}
\begin{lemma}
\label{lemma_determinant}
Under the assumptions of the previous lemma and corollary, fix some $\theta'\in \T^{d-1}\setminus \tilde{S}$ and $u\in \{-1,1\}^{d-1}$. For $t>0$, consider the following $1$-periodic subset of $\C$:
$$
\Theta_t=\{\theta_1\in \C\colon |2\cos 2\pi (\theta_1+b_1)+F(\theta'+b'-itu)|\ge e^{\pi t/2},\,\forall b=(b_1,b')\in \Lambda'\}.
$$
There exists $t_0=t_0(\theta')$ such that, for $t>t_0$, every connected component of $\C\setminus\Theta_t$ is bounded and contains at most one root of the equation
\beq
\label{eq_cor_determinant}
\det(\Delta(\eta,\theta'-itu))=0,
\eeq
considered as an equation on $\eta\in \C$, and that root is simple.
\end{lemma}
\begin{proof}
Choose a large $t_0$ so that the conclusion of Corollary \ref{cor_separation} holds for the given choice of $u$, and suppose that $t>t_0$. For these $t$ and every $\eta\in \C\setminus \Theta_t$, Lemma \ref{lemma_separation} implies that exactly one diagonal entry of $\Delta(\theta_1,\theta')$ is smaller than $e^{\pi t/2}$ in absolute value. In other words, for any $\eta\in \C\setminus\Theta_t$ there exists a unique $b=b(\eta)=(b_1,b')\in \Lambda'$ such that
\beq
\label{eq_all_gamma}
|2\cos 2\pi (\eta+b_1)+F(\theta'+b'-itu)|\le e^{\pi t/2}.
\eeq
Let $W$ be a connected component of $\C\setminus\Theta_t$. Clearly, $b(\eta)$ is constant on $W$. The first claim of Corollary \ref{cor_separation} implies, perhaps after choosing a larger $t_0$, that $W\cap \R=\varnothing$ and
\beq
\label{eq_cosine_w}
|2\cos2\pi(\eta+b_1)|\ge \frac12 e^{3\pi t/2},\quad \forall\eta\in W.
\eeq
We will use these estimates to obtain a bound on the size of $W$. Suppose that 
$$
\eta_1=x+i y_1\in W,\,\,\eta_2=x+1/2+i y_2\in W;\quad x,y_1,y_2\in \R.
$$
Since $W$ does not intersect the real line, we can assume without loss of generality that $y_1,y_2>0$. From \eqref{eq_cosine_full} and \eqref{eq_cosine_w}, it follows that
$$
|\cos 2\pi(\eta_1+b_1)-\cos 2\pi (\eta_2+b_1)|\ge e^{\pi t}.
$$
Together with the triangle inequality, this contradicts \eqref{eq_all_gamma} for large $t$. As a consequence, $W$ cannot contain two points whose real parts differ by $1/2$, and therefore is contained in a strip of width $1/2$. Due to $\eqref{eq_cosine_full}$ and the fact that $F(\theta'+b'-itu)$ does not depend on $\theta_1$, it is also clear that $W$ is contained in a horizontal strip, say, $|\Im\eta|\le 20t$, and therefore is bounded.

Assume that there are two roots $\eta_1,\eta_2$ of \eqref{eq_cor_determinant} in the same connected component of $\C\setminus \Theta_t$. In view of the above, we have
$$
0=\cos 2\pi(\eta_1+b_1)-\cos 2\pi(\eta_2+b_1)=-2\sin\pi(\eta_1-\eta_2)\sin\pi(\eta_1+\eta_2+b_1).
$$
Since $W\cap \R=\varnothing$ and $W$ is connected we have $\Im(\eta_1+\eta_2)\neq 0$, and therefore the second factor cannot vanish. Since $W$ is contained in a vertical strip of width $1/2$, the first factor can only vanish for $\eta_1=\eta_2$.
\end{proof}
\begin{cor}
\label{cor_main}
Under the assumptions of Lemma $\ref{lemma_separation}$, suppose that $\theta'\in \T^{d-1}\setminus \tilde S$, where $\tilde S$ is defined in Corollary $\ref{cor_separation}$, and $E\in \sigma(H)$. Then the set
$$
\{t\in \C\colon \det(h(\cdot,\theta'-itu)-E)=0\,\text{ has a degenerate root}\}
$$
is discrete in $\C$.
\end{cor}
\begin{proof}
From Proposition \ref{prop_discriminant}, it is sufficient to produce a single value of $t\in \C$ such that all roots of the above equation are simple. Since $E\in \sigma(H)$, we have $|E|\le 2d+\|V\|_{\infty}$. Apply Lemma \ref{lemma_determinant} and choose $t$ satisfying $e^{\pi t/2}>2d+2\|V\|_{\infty}$ in addition to the choices already made. For these $t$, we have
$$
\det(\Delta(\theta_1,\theta'-itu)+s(V-E))\neq 0,\,\forall \theta_1\in \Theta_t,\,\forall s\in [0,1].
$$
Let $f_s(\eta):=\Delta(\eta,\theta'-itu)+s(V-E))$. Clearly, $f_s(\eta)$ is holomorphic on $\C$ and, for any connected component $W$ of $\C\setminus\Theta_t$, we have $|f_s(\eta)|>c>0$ on $\partial W$ (here we used the fact that $W$ is bounded from the previous lemma). Since $\partial W$ is a piecewise smooth Jordan curve, the number of zeros of $f_s$ in $W$ is equal to
$$
\frac{1}{2\pi i}\oint_{\partial W}\frac{f_s'(\eta)}{f_s(\eta)}\,d\eta
$$
and is continuous and therefore constant in $s$. By considering $s=0$ and using the previous lemma, we see that this integral can only be equal to $0$ or $1$. As a consequence, each root of the equation $\det(h(\theta_1,\theta'-itu)-E)=0$, considered as an equation on $\theta_1$, is simple.
\end{proof}
\subsection{Proof of Theorem \ref{th_main_edges}} Let $E=E_j^{\pm}$ and consider the level set
$$
L=\{\theta\in \R^d\colon E_j(\theta)=E\}.
$$
Assume that $\Gamma$ is not divisible. Without loss of generality (perhaps, after switching the roles of the components of $\theta$), one can assume that $\Gamma$ satisfies the assumptions of Lemma \ref{lemma_separation}. From Proposition \ref{prop_ac}, it also follows that the intersection of $L$ with any line of the form $\theta'=\mathrm{const}$ is finite. As a consequence, it would be sufficient to show that, under the assumptions of Lemma \ref{lemma_separation}, we have $\dim L'\le d-2$, where
$$
L'=\{\theta'\in \R^{d-1}\colon \exists \,\theta_1\in \R\,\text{ such that }\,(\theta_1,\theta')\in L\}
$$
is the projection of $L$ onto the $\theta'$ hyperplane.

Let $u\in \{-1,1\}^{d-1}$ be the vector from Lemma \ref{lemma_separation}, and consider the decomposition
$$
\theta'=su+w,\quad w\in W:=\R^{d-1}\cap \{u\}^{\perp},\quad s\in \R.
$$
From Proposition \ref{prop_discriminant}, it follows that $W=W_1\cup W_2$, where
$$
W_1=\{w\in W\colon \det(h(\cdot,w+su)-E)=0\,\text{ has a degenerate root for all }s\in \C\}
$$
and
$$
W_2=\{w\in W\colon \det(h(\cdot,w+su)-E)=0\,\text{ has a degenerate root for a discrete set of  }s\in \C\}.
$$
Clearly, the union is disjoint. Similarly, for the variable $\theta'$ we have $\R^{d-1}=U_1\cup U_2$, where
$$
U_1=W_1+\R u,\quad U_2 =W_2+\R u.
$$
We will show that $\dim U_1\le d-2$ and $\dim (L'\cap U_2)\le d-2$. 
From Corollary \ref{cor_main}, for any $w\in W_1$ and $s\in \R$, the vector $su+w$ cannot belong to the complement of $\tilde S$ (see Corollary \ref{cor_separation}) and therefore must belong to $\tilde S$. Therefore, $U_1\subset\tilde{S}$ and $\dim U_1\le d-2$. For each $w\in W_2$, a vector $\theta'=su+w$ can only belong to $L'$ if the equation $\det(h(\theta_1,su+w)-E)=0$ has degenerate roots (as an equation in $\theta_1$; see Proposition \ref{prop_degenerate_root}). By the definition of $W_2$, this can only happen for a discrete set of values of $s\in \R$, and therefore $L'\cap U_2$ has discrete intersection with any line parallel to $u$. This implies $\dim (L'\cap U_2)\le \dim W_2\le d-2$ and completes the proof.\,\qed
\subsection{On some divisible lattices}
First, consider the case $p=2$, and let $V$ be a checkerboard potential. It is easy to see that the operator $h(\theta)$ is now unitarily equivalent to
$$
\begin{pmatrix}
v_0 & \cos(\theta_1)+\ldots+\cos(\theta_d)\\
\cos(\theta_1)+\ldots+\cos(\theta_d)& v_1
\end{pmatrix}.
$$
As a consequence, the internal edges of the spectral bands are defined by the equation $\cos(\theta_1)+\ldots+\cos(\theta_d)=0$ and therefore the corresponding Fermi surfaces have dimensions $d-1$.

Regarding periodic potentials with larger periods, the following result is established in \cite[Remark 4.1]{PS}.
\begin{prop}
\label{prop_ps}
Let $d=2$, $p\ge 3$, the lattice $\Gamma$ be spanned by $\{(1,1),(p,0)\}$ and $v(n_1,n_2)=V_{(n_1+n_2)\,\mathrm{mod}\,p}$ be a real-valued $p$-periodic potential. Assume that $V_0<V_j-2$ for $j=1,2,\ldots,p-1$. Then the right edge of the left-most spectral band of the corresponding discrete Schr\"odinger operator is equal to $V_0$, and the corresponding Fermi surface has dimension one.
\end{prop}
As a consequence, for any divisible lattice $\Gamma\subset \Z^2$, there exists a $\Gamma$-periodic potential such that the conclusion of Theorem \ref{th_main_edges} does not hold.

\end{document}